\documentclass[12pt]{amsart}
\usepackage[margin=1.2in]{geometry}
\usepackage{amsmath}
\usepackage{amssymb}
\usepackage{amsfonts}
\usepackage{mathrsfs}
\usepackage{amsthm}
\numberwithin{equation}{section}
\usepackage[all]{xy}
\usepackage{graphicx}
\usepackage{hyperref}
\usepackage{verbatim}
\hypersetup{
    colorlinks,
    citecolor=red,
    filecolor=black,
    linkcolor=blue,
    urlcolor=black
}

\def \IN{\mathbb N}

\def \IZ{\mathbb Z}

\def \IR{\mathbb R}

\newcommand{\art}{\ar@{-}}

\newtheorem{thm}{Theorem}[section]
\theoremstyle{definition} \newtheorem{defin}[thm]{Definition}
\theoremstyle{plain} \newtheorem{lemma}[thm]{Lemma}
\theoremstyle{plain} \newtheorem{prop}[thm]{Proposition}
\theoremstyle{remark} \newtheorem{rem}[thm]{Remark}
\theoremstyle{plain} \newtheorem{cor}[thm]{Corollary}
\theoremstyle{remark} 

\newcommand{\trm}[1]
{\textrm{#1}}

\newcommand{\SL}{\textrm{SL}}

\newcommand{\Orb}{\mathcal{O}}
\newcommand{\orb}{\mathfrak{o}}

\renewcommand{\d}[1]{\ensuremath{\operatorname{d}\!{#1}}}

\newcommand{\GSp}{\textrm{GSp}}

\newcommand{\mat}[4]
{\left(\begin{array}{cc}#1 & #2 \\ #3 & #4\end{array}\right)}

\newcommand{\Sp}
{\trm{Sp}}
\newcommand{\SP}
{\mathbb{S}}
\newcommand{\prim}
{\textrm{prim}}
\newcommand{\vol}
{\trm{vol}}
\newcommand{\tr}
{\intercal}
\newcommand{\diag}
{\trm{diag}}
\newcommand{\stab}
{\trm{Stab}}
\author{Jayadev S. Athreya and Ioannis Konstantoulas}
\title{Discrepancy of general symplectic lattices}
\begin{document}
\maketitle
\begin{abstract}
We consider random lattices taken from the general symplectic ensemble and count the number of lattice points of a typical lattice in nested families $B_t$ of certain Borel sets.  Our main result is that for almost every general symplectic lattice, the discrepancy $D(\Lambda,B_t)$ of the lattice point count with respect to the volumes is $O(\vol(B_t)^{-\delta})$.  This extends work of W. Schmidt who gave similar discrepancy bounds for the space of all lattices in $\IR^n$.
\end{abstract}
\section{Introduction}
\label{sec:introduction}

Let $V$ be a $2n$-dimensional vector space over $\IR$, $\sigma = \langle\mbox{ },\mbox{ }\rangle$ a symplectic form on $V$ and $\{e^i,f^i\}_{(i=1,\cdots n)}$ a symplectic basis for $\sigma$.
We will write $v\in V$ as $v = (x,y)$ where $x= \sum x_ie^i$ and $y=\sum y_if^i$ with respect to this basis.
By Darboux's theorem, we may take $\sigma$ to be the standard symplectic form and $e^i,f^i$ the standard basis of $\IR^{2n}$.  
We will denote the inner product on $V$ by $[\mbox{ },\mbox{ }]$ and define the linear map $\tau(x,y) = (y,-x)$; then $\langle v^1,v^2\rangle = [v^1,\tau(v^2)]$.

It is well known that $G =\Sp(2n,\IR)$ acts transitively on symplectic bases of $V$ which can then be identified with the columns of a $g\in G$.  
Thus from a given $g$ we get a symplectic lattice $\Lambda_g = g\IZ^{2n}$ and all such lattices are obtained this way.
The stabilizer of $\IZ^{2n}$ is $\Gamma = \Sp(2n,\IZ)$, giving the space of symplectic lattices the structure of a finite volume homogeneous space $\SP= \SP_V = G/\Gamma$.
Given a Borel set $B\subset V$, we are interested in counting the number of points in $\Lambda_\prim\cap B$ when $\Lambda$ is picked randomly according to the probability Haar measure from $\SP$ and $\Lambda_\prim$ is the set of primitive points in $\Lambda$.

By a result in \cite{MoSa} inspired by Siegel's theorem from \cite{Sieg} we know that the average number of primitive lattice points is 
\begin{equation}
  \label{eq:1}
  \int_{\SP} \sum_{\lambda\in \Lambda_{\prim}}\chi_B(\lambda)\, d\mu(\Lambda) = \frac{\trm{vol}(B)}{\zeta(2n)}.
\end{equation}
Indeed, by the transitivity of the action of $\Sp(2n,\IR)$ on non-zero vectors, one can exhibit $\IR^{2n}$ as a homogeneous space and Lebesgue measure as the projected Haar measure on the space.
The normalizing constant comes from giving $\SP$ measure $1$ and the fact that the density of primitive integral vectors in $\IZ^{2n}$ is $\frac{1}{\zeta(2n)}$.

From the probabilistic point of view, it is then natural to ask about the variance around this mean:
\begin{equation}
  \label{eq:2}
  V(B) = \left|\int_{\SP} \left(\sum_{\lambda\in \Lambda_{\prim}}\chi_B(\lambda)\right)^2\, d\mu(\Lambda)-\left(\frac{\trm{vol}(B)}{\zeta(2n)}\right)^2\right|.
\end{equation}

For the case $n=1$ and for the full space of unimodular lattices $\SL(m,\IR)/\SL(m,\IZ)$, square root error estimates have been provided using Rogers's formulas (which exhibit the moments of the Siegel transform of a function in terms of data related to the moments of the function itself) and the machinery of Eisenstein series on $\SL(2,\IR)/\SL(2,\IZ)$.  Recently significant progress has been made in the situation of lattices taken from the rank $1$ homogeneous spaces $\trm{SO}(n,1)(\IR)/\trm{SO}(n,1)(\IZ)$ by Shucheng Yu in \cite{Yu} again using spectral analysis of certain Eisenstein series for those groups.

Rogers's formulas seem to be specific to the full space of lattices and the Eisenstein series approach, although promising for $\SP$, are much harder to work with than the rank $1$ case.
In this work, we prove a Rogers type formula for the symplectic group and apply it to give estimates for the corresponding variance $V(B)$ over a slightly bigger space, namely the unit cone over $\SP$ in the infinite measure space $\GSp^+(2n,\IR)/\GSp^+(2n,\IZ)$.   After we completed our work, Kelmer and Yu in \cite{KelYu} used the representation theory of $\Sp(n,\IR)$ to give $L^2$ norm bounds for the non-primitive symplectic lattice point count; their discrepancy results are similar to ours but utilize very different tools. 

To state our main result we need the notion of a Siegel transform, defined in \ref{Siegel-def} as the sum of the input function over all primitive lattice points of a given lattice.  We also need the following data:
\begin{defin}
For positive integers $s$ and $d|s$ let $q=\frac{s}{d}$ and $$a(s) = \frac{1}{\zeta(2n)s^{2n-1}}\sum_{d|s}\phi(d)q^{2m+1}\prod_{p|q}\left(1-\frac{1}{p^{2m}}\right).$$  For arbitrary integer $s$ let $$G(s) = \int_{\IR^{2n}}\int_{\IR^{2n-1}}f(x)f\left(sy^*+\sum_i t_iy_i(x)\right)\,dt\,dx$$ where $y^*=y^*(x)$ satisfies $\langle x,y^*(x)\rangle=1$ and the $y_i$ span the hyperplane $\langle x,y\rangle=0$.
  \end{defin}
  Our Rogers type formula takes the form

  \begin{thm}
 For any integrable $f$ on $\IR^{2n}$ we have
 $$\int_{S_n}\widehat{f}^2(\Lambda)d\Lambda = \sum_{s=-\infty}^\infty a(|s|) G(s) + \sum_{(k,l)=1}\int_{\IR^{2n}}f(kx)f(lx)\,dx.$$
 
  \end{thm}

The main application we consider in this paper is to estimate the discrepancy of primitive lattice points

\begin{equation}
  \label{eq:3}
D(B,\Lambda) = \left|\frac{(\#\Lambda_\prim\cap B)}{\vol(B)} - \frac{1}{\det(\Lambda)\zeta(2n)}\right|  
\end{equation}
for almost every $\Lambda$ in $\GSp(2n,\IR)/\GSp(2n,\IZ)$ (with respect to any measure equivalent to Haar).

This result is in the spirit of the seminal work by Schmidt (\cite{Schm}) on discrepancy estimates for the space of all lattices.
Schmidt uses his investigation of \eqref{eq:2} to obtain the bound $$D(B_m,\Lambda)=O(\vol(B_m)^{-1/2}\log^k(\vol(B_m))$$ for almost every lattice $\Lambda$ in $\textrm{GL}(n,\IR)/\textrm{GL}(n,\IZ)$, where $B_m$ is a nested sequence of Borel sets of volumes approaching infinity.  In fact, Schmidt shows that such a sequence can be embedded in a continuous family $B_t$ with the property that the values $\vol(B_t)$ range over all non-negative real numbers; here we will make this assumption at the onset in order to simplify our presentation, although our results hold more generally.  

The families $B_t$ to which our results are applicable are restricted because of the way certain symplectic integrals need to be bounded.  As a special case of our results, we give the analogue of Schmidt's result for a continuous family of sets $B_t$, $t>0$ with the property that the bulk of their volumes $v_t$, up to $v_t^{\epsilon'}$ for some uniform $\epsilon'$, lie in symplectic ellipsoids of the form $g_tB(0,v_t^{\frac{1}{2n}})$ for a family $g_t\in\Sp(2n,\IR)$ and the function $t\mapsto v_t$ has range $[v_0,\infty)$ for some $v_0\geq 0$.

More generally, the condition we need to impose on the Borel sets $B_t$ is that for every $\delta>0$ there exists an $\epsilon > 0$ such that
\begin{equation}
  \label{eq:37}
  \int_{B_t}\int_{B_t}|\langle x,y\rangle|_+^{-\delta}dx\,dy \leq Cm(B_t)^{2-\epsilon}.
\end{equation}
where $|A|_+ = \max\{1,|A|\}$.

As long as this condition is satisfied, the family $B_t$ is admissible for our bounds and we get a power savings with exponent depending explicitly on $\epsilon$ and $\delta$.  It must be noted that for $n=1$ all Borel sets satisfy this inequality; it is not clear to us if this holds true in higher dimensions as well, but for any natural family of sets this hypothesis holds with explicit $\epsilon$.

Our main theorem then states:
\begin{thm}
  Let $B_t$ be a continuous family of Borel sets satisfying \eqref{eq:37}.  Then there exists $\delta'>0$ depending only on the dimension and $\epsilon$ above such that for Haar-almost every\footnote{Almost every makes sense for any equivalent measure to the Haar on $\GSp(2n,\IR)/GSp(2n,\IZ)$.} lattice $\Lambda\in \GSp^+(2n,\IR)/\GSp(2n,\IZ)$, we have
  \begin{equation}
    \label{eq:41}
    D(B_t,\Lambda)\leq C(\Lambda, n,\delta')\vol(B_t)^{-\delta'}
  \end{equation}
for all sufficiently large $t\geq t_0(\Lambda)$.
\end{thm}

Although we are inspired by Schmidt's work and follow his strategy, the analysis is considerably more complicated in the case of $\Sp(2n,\IR)$ for $n>1$; orbits have complicated stabilizers which necessitates use of the analytic continuation of Dirichlet series of multiplicative functions and some group theoretic considerations.
We begin by looking at the second moment of the Siegel transform of $f$, and using an unfolding technique we are led to classify the structure of $\Sp(2n,\IZ)$-orbits on $(\IZ^{2n}_\prim)^2$.  We then link those to certain congruence subgroups of $\Sp(2n,\IZ)$ whose indexes are involved in a multiplicative function that we can control on average.  We then decouple the geometric part of the second moment from the arithmetic, analyze them both and combine the pieces of information to give the bound.

\section{Analysis of the second moment}
\label{sec:analys-second-moment}

In order to extract the main and error terms from the second moment of the (primitive) Siegel transform
\begin{equation}\label{Siegel-def}
  \widehat{f}(\Lambda) = \sum_{\lambda\in\Lambda_\prim}f(\lambda)
\end{equation}
we will regroup the sum into orbits of the diagonal action of $\Sp(2n,\IZ)$ on primitive pairs of integral vectors.
We will use the straightforward steps of the regrouping as an opportunity to introduce notation and make our conventions clear.

Let $F$ be a (fixed once and for all) fundamental domain of $\Sp(2n,\IR)$ for the action of $\Sp(2n,\IZ)$, $d\mu$ the Haar measure that gives total measure $1$ to $F$.
We will occasionally be sloppy and write $\Sp(2n,\IR)/\Sp(2n,\IZ)$ for $F$; any homogeneous space that we will write down will be either a discrete tiling by $F$ with the obvious counting measure on the quotient, or an embedded subdomain in $F$ with the normalized Haar measure induced by restricting the Lie algebra directions to the homogeneous space defining the subdomain.

Now we can proceed to expand and regroup the square of the Siegel transform:
\begin{eqnarray*}
 \left(\widehat{f}(g\IZ^{2n})\right)^2&=&\left(\sum_{z\in\IZ^{2n}_{\trm{\prim}}}f(g z)\right)^2\\ &=& \sum_{(u,v)\in\IZ^{2n}_{\trm{\prim}}\times \IZ^{2n}_{\trm{\prim}}}f(gu)f(gv)\nonumber\\
&=& \sum_{\orb\in \Orb}\sum_{(u,v)\in \orb}f(gu)f(gv)\\
&=& \sum_{\orb\in\Orb}\sum_{\gamma \in \Sp(2n,\IZ)/\stab_\Gamma(\orb)}f(g\gamma u_0)f(g\gamma v_0).
\end{eqnarray*}
Here $\Orb$ denotes the set of orbits of $\Gamma=\Sp(2n,\IZ)$ acting diagonally on $\IZ^{2n}_\prim\times \IZ^{2n}_\prim$ and  $\stab_\Gamma(\orb)$ is the stabilizer in $\Gamma$ of an arbitrary representative $(u_0,v_0)$ of the orbit $\orb$ (we will pick convenient representatives later).

Now we pass the integration over the fundamental domain $F$ down to the level of individual orbits to get
\begin{eqnarray*}
 \int_F\left(\widehat{f}(g\IZ^{2n})\right)^2\,d\mu(g)&=& \sum_{\orb\in\Orb}\int_F\sum_{\gamma \in \Sp(2n,\IZ)/\stab_\Gamma(\orb)}f(g\gamma u_0)f(g\gamma v_0)\,d\mu(g)\\
&=& \sum_{\orb\in\Orb}\int_{U}f(g u_0)f(g v_0)\,d\mu(g),
\end{eqnarray*}
where $$U=\bigcup_{\gamma\in \Sp(2n,\IZ)/\stab_\Gamma(\orb)}F\gamma$$
is the fundamental domain for the infinite measure homogeneous space $Sp(2n,\IR)/\stab_\Gamma(\orb)$ given by tiling the translates of $F$.

Thus the problem of estimating this second moment breaks down to understanding the set $\Orb$ of diagonal orbits of $\Gamma$ on primitive pairs of vectors and the individual contributions of each term.
The first part is dealt with in the next section, while the second part is treated in the subsequent sections where it splits further into ``intrinsically'' arithmetic contributions and geometric contributions.
The former is reflected in the indices of conjugates of $\stab_\Gamma(\orb)$ in $\stab_\Gamma(\orb_1)$ for a specific orbit $\orb_1$ and the latter in the fiber measure of an embedded copy of $\Sp(2n,\IR)/\Sp(2n-2,\IR)$ in $U$.
Happily, the two conspire to give exactly the expected behavior as we average over orbits.
\section{Diagonal orbits of $\trm{Sp}(2n)$}
\label{sec:orbits}

Here we describe the structure of orbits of $G$ and $\Gamma$ on $V=\IR^{2n}$, $V^{\otimes 2}$ and the lattices $\IZ^{2n}_\prim$ (in the standard coordinates of $\IR^m$) and its tensor square.
For a given $g\in G$ we will write 
\begin{equation}
  \label{eq:4}
  g = \mat{a}{b}{c}{d}=\mat{a_g}{b_g}{c_g}{d_g}
\end{equation}
where $a,b,c,d$ are $n\times n$ matrices satisfying the symplectic equations
\begin{eqnarray*}
  \label{eq:5}
  a^\tr d - c^\tr b &=& \trm{I}\\
  a^\tr c &=& (a^\tr c)^\tr\\
  b^\tr d &=& (b^\tr d)^\tr.
\end{eqnarray*}
\begin{lemma}
  $G$ acts transitively on $V\setminus \{0\}$ and on level sets $\sigma(v^1,v^2)=s$ of $V\times V$ with linearly independent pairs.
\end{lemma}
\begin{proof}
  The first claim and $s=0$ of the second follow from the transitivity of $G$ on isotropic subspaces of a given dimension.
For $s\neq 0$, the span $V_0$ of two independent vectors $v^1,v^2$ is non-degenerate, and $v^1,\frac{v^2}{s}$ is a symplectic basis for $(V_0,\sigma_{V_0})$.
It is easy to show that $V_0^\perp$ is also non-degenerate, thus admits a symplectic basis $w^1,\cdots, w^{2n-2}$ for $\sigma_{V_0^\perp}$.
By orthogonality with respect to $\sigma$, $v^1,\frac{v^2}{s},w^1,\cdots$ are a symplectic basis of $(V,\sigma)$.
Then the matrix with columns $v^1,\cdots,v^2/s,\cdots$ is a symplectic matrix mapping $e^1$ to $v^1$ and $f^1$ to $v^2/s$.
\end{proof}
The next step is to understand the orbits of $\Sp(2n,\IZ)$.
\begin{lemma}\label{lem:trans}
  The group $\Sp(2n,\IZ)$ acts transitively on primitive vectors in $\IZ^{2n}$.
\end{lemma}
\begin{proof}
  See \cite[Section 5.1]{MoSa}.
\end{proof}
\begin{rem}
When we found the explicit reference \cite{MoSa}, we realized that their method for proving the lemma had some things in common with the method we use to describe orbits of pairs.
Since our proof of the latter is a little involved and the ideas are similar, we recommend the reader to start with the proof of transitivity given there and come back to the orbit classification here.
\end{rem}
\begin{defin}
We call a pair of vectors $v^1,v^2\in \IZ^{2n}$ primitive if both vectors are primitive and linearly independent.
\end{defin}

\begin{prop}
\label{prop:orbits}
  The orbits of the diagonal action of $\Sp(2n,\IZ)$ on primitive pairs are in bijection with triples $(s,d,a)$ where $s=\langle v^1,v^2\rangle$, $d>1\neq 0$ is a divisor of $s$ and $a\in (\IZ/d\IZ)^*$, the triple $(s,1,0)$ and triples $(s,0,a)$ where $a\in (\IZ/s\IZ)^*$.
\end{prop}
\begin{proof}
First of all, it is clear that $s$ is preserved by symplectomorphisms.
Let $d$ be the greatest common divisor of all $2\times 2$ minors of the $(2n)\times 2$ matrix $(v^1,v^2)$;
to see this is preserved, note that the diagonal action of $\Gamma$ gives a linear action on $V^{\otimes 2}$  where it preserves decomposable tensors and the subspace of symmetric tensors, thus descending to $\bigwedge ^2 V$.
Since the coefficients of a decomposable $2$-form are the $2\times 2$ minors of the corresponding matrix and $\Gamma$ is a linear action over $\IZ$ on $\bigwedge^2 V$, the equation $g\cdot v^1\wedge v^2 = g v^1\wedge g v^2$ shows that the $2\times 2$ minors of the image are integer linear combinations of those of $v^1\wedge v^2$, and thus the $\gcd$ is non-decreasing.
Since the action is invertible, the $\gcd$ is non-increasing either, so it is preserved.

For the remaining datum characterizing the orbit, and to show transitivity within this data, using coordinates is more convenient.  As before, we write $v^i = (x^i,y^i)$.
Using Lemma \ref{lem:trans}, we can assume $v^1=e^1$ and write $v^2=v$ so that $s=v_{n+1} = y_1$. 
First assume $s>0$ and the case $s<0$ will follow.
Consider a block-diagonal matrix $g = \mat{g_1^*}{0}{0}{g_1}$ in $\Sp(2n,\IZ)$ where $g^*$ is the inverse transpose and the first row (resp. column) of $g_1$ is $e^1$ (resp $(e^1)^\tr$).
The remaining $(n-1)\times (n-1)$ submatrix is an arbitrary element of $\SL(n-1,\IZ)$; note $g$ fixes $e^1$.
Let $d_0= \gcd(y_2,\cdots,y_n)$.
Then $d_0^{-1}(y_2,\cdots,y_n)$ is either primitive or, if $d_0=0$, the zero vector.  
In any case, by the transitivity of $\SL(n-1,\IZ)$ on primitive vectors we can take $(y_2,\cdots, y_n)$ to $(0,\cdots, d_0)$.

Next, we need elements of $\Sp(2n,\IZ)$ of the form 
$$ u_s=\mat{\trm{I}}{s}{0}{\trm{I}},\quad l_s =   \mat{\trm{I}}{0}{s}{\trm{I}}$$
where $s$ is an $(n\times n)$ symmetric matrix.
Suppose $d_0> x_n$ and $d_0\neq 0$. 
Using $u_s$ with $s = qE_{nn}$ ($E_{nn}$ is the elementary matrix with $1$ at the $(n,n)$-th entry and zero elsewhere), we can replace $x_n$ with the remainder after division by $d_0$, leading to $x_n\mapsto r$ with $r<d_0$.
If $d_0<x_n$, which will be the case after the previous step, we can do the same with the appropriate $l_s$.
Thus by alternating between suitable $l_s$ and $u_s$ we implement the Euclidean algorithm to turn the pair $\{x_n,d_0\}$ into $\{d_1,0\}$ where $d_1 = \gcd(d_0,x_n)$.
Once we are done, if the new $y_n = 0$ then we switch the two with one more pair of $u_s$,$l_s$.

Now we want to iterate this procedure and use $d_1$ to turn $x_{n-1},x_{n-2},\cdots,x_2$ into $0$ and $d_1$ into $d_2,\cdots d_{n-3} =d'$ where $d'$ will be the greatest common divisor of all the $2\times 2$ minors except the one involving $s$.
For this, we will need to use $l_s$ with $s$ having off-diagonal entries, so care must be taken not to ruin the previous arrangement.
However, to carry out the division process between $x_{n-1}$ and $d_1$, the matrix $s$ in the $u_s$ we need  has non-zero entries $(n-1,n)$ and $(n,n-1)$; the former performs the division, and the latter does not affect the zero at the $x_n$ entry, since $y_{n-1}=0$ after the reduction step.
The same argument shows the $l_s$ will not ruin the zero entries in the pair.
This process can be iterated to eliminate all non-zero entries $x_2,\cdots, x_n$ and bring the pair to the form 
\begin{equation}
  \label{eq:6}
  (e^1, ae^1 + sf^1 + d'f^n) = \left(\begin{array}{cc}1 & a \\ 0 & 0\\ \vdots & \vdots \\ 0 & s\\ \vdots & \vdots\\ 0 & d'\end{array}\right). 
\end{equation}

If $d_0=0$, so all the $y_i,\, i>1,$ are zero, then we use a second block-diagonal matrix $g'=\diag(g_2,g_2^*)$ fixing $e^1$ to transform the vector $(x_2,\cdots,x_n)$ into $(0,\cdots, d')$, where now we suppose $d'\neq 0$.
Then as before we switch the zero at the $y_n$-th entry with $d'$ and we are back at the previous form.
If $d'$ is also zero, we are in the orbit of $(a,0\cdots,s,0\cdots)^\tr$

The next step is to replace $d'$ by $d=\gcd(d',s)$; we will write down the matrices in the cases $n=2,\,3$; the general case follows by iterated applications of these steps to successively eliminate the non-zero coordinates of the second vector.

Let $k,l$ be integers such that $sk+ld' = d$ and let $d=rd'$.
For $n=2$, we have
\begin{equation}
  \label{eq:7}
  \left(\begin{array}{c} a+kd' \\ 0 \\ s \\ d\end{array}\right)
=
  \left(\begin{array}{cccc}1 & 0 & 0 & k\\ 0 & r & rk & rl-1\\ 0 & 0 & 1 &0\\ 0 & 1 & k & l \end{array}\right)
  \left(\begin{array}{c} a \\ 0 \\ s \\ d'\end{array}\right).
\end{equation}
For $n=3$ pick
\begin{equation*}
  M_1 = \left(\begin{array}{cccccc}1 & 0 &0&0&k&0\\0&0&0&0&0&-1\\0&0&0&0&-1&0\\0&0&0&1&0&0\\0&0&1&0&l&0\\
0&1&0&k&0&l \end{array}\right)
\end{equation*}
which transforms
\begin{equation}
  \label{eq:9} 
  \left(\begin{array}{c} a \\ 0 \\ 0 \\ s\\0\\d'\end{array}\right) \mapsto \left(\begin{array}{c} a \\ -d' \\ 0 \\ s\\0\\d\end{array}\right)
\end{equation}
and then pick $M_2 = u_s$ where $s = rE_{32}+rE_{23}$ to eliminate the new non-zero entry.
Then $M=M_2M_1$ provides the desired transformation (this is how we built the matrix for $n=2$ as well).

Note that $d$ is indeed the second elementary divisor of the matrix in \eqref{eq:6} and we already saw that it is preserved by symplectic maps (this can now also be seen by the representative we got).
It remains to show that two pairs $(s,d,a),\,(s,d,a')$ are in the same orbit if and only if $a\equiv a'\pmod{d}$
 (by primitivity and the form to which we brought each pair, $\gcd(d,a)=1$).
But now it is clear that if they are in the same orbit, then $a' = a + xs + yd$ for some $x,y\in \IZ$, so the claim follows.
\end{proof}

\section{Algebraic structure of the stabilizers}
\label{sec:struct-stab}
In order to analyze the integral in \eqref{eq:2}, we want to split it into orbits and reassemble the pieces into orbital integrals for the action of $\Sp(2n,\IR)$.
Fix an orbit of $\Sp(2n,\IZ)$ represented by $(s,d,a)$; we need to understand the groups $S_G=\stab_G((s,d,a)_0)$ and $S_\Gamma=\stab_\Gamma((s,d,a)_0)$ where $(s,d,a)_0$ denotes the explicit representative we found in Proposition \ref{prop:orbits} and we will denote it also by $(s,d,a)$ from now on.  To simplify the appearance of formulas below, let $m=n-1$.

Let $(s)$ denote the pair $\langle e^1,sf^1\rangle$.
It is clear that the stabilizer of $(s)$ in $\Sp(2n,\IR)$ is the embedded $\Sp(2m,\IR)=: S(s)$ fixing the symplectic plane $\IR e^1+\IR f^1$, and this describes the stabilizer in $G$ of any point of an orbit $(s,d,a)$ as a suitable $G$-congugate of $S(s)$ (in fact, as a conjugate over $\Sp(2n,\IZ[\frac{1}{s}])$).

Explicitly, the transformation in  $\Sp(2n,\IZ[\frac{1}{s}])$ given by
\begin{equation}
  \label{eq:10}
  T_{(s,d,a)} = \left(e^1,e^2,\cdots,\frac{-d}{s}e^1+e^n, \frac{a}{s}e^1+f^1 + f^n\frac{d}{s}, f^2,\cdots,f^n\right)
\end{equation}

fixes $e^1$ and takes $(a,0,\cdots, s,\cdots, d)^\tr$ to $sf^1$.
Thus we have $$ \stab_G(s,d,a) = T_{(s,d,a)}^{-1} S(s) T_{(s,d,a)}.$$

The crucial thing about this stabilizer is that its $\IZ$-points come from the subgroup of $\Sp(2m,\IZ)$ of matrices with the constraints ($i$ below ranges from $1$ to $m-1$ for $n\geq 3$):

\begin{eqnarray}\label{congr}
  \mat{A}{B}{C}{D} &:& A_{mi}\equiv D_{im}\equiv 0 \pmod{q}\\ & &A_{mm}\equiv D_{mm}\equiv 1\pmod{q},\nonumber\\
& & B_{mi}\equiv B_{im}\equiv 0\pmod{q}\nonumber\\ & & B_{mm}\equiv 0\pmod{q^2}.\nonumber
\end{eqnarray}
where $q=\frac{s}{d}$.

For instance, when $n=2$ the stabilizer is conjugate to the preimage of
\begin{equation}
  \label{eq:13}
  \left\{\mat{1+qk}{0}{l}{1-qk}: k\in\IZ/q\IZ,\, l\in\IZ/q^2\IZ \right\}<\SL(2,\IZ/q^2\IZ)
\end{equation}

of cardinality $q^3$, so the stabilizer corresponds to a subgroup of $\SL(2,\IZ)$ of index 
\begin{equation}
  \label{eq:14}
  [\SL(2,\IZ):S_\Gamma(s,d,a)]=q^3\prod_{p|q^2}\left(1-\frac{1}{p^2}\right).
\end{equation}

In general, we have
\begin{lemma}
  The index of $S_\Gamma(s,d,a)$ in the copy of $\Sp(2m,\IZ)$ in $S(s)$ is 
  \begin{equation}
    \label{eq:11}
    q^{2m+1}\prod_{p|q}\left(1-\frac{1}{p^{2m}}\right)
    \end{equation}
where the outer product is over primes $p$ dividing $q$.
\end{lemma}
\begin{proof}
  First of all we have the general formula for $q\geq 2$
  \begin{equation}
    \label{eq:12}
    \left|\Sp(2n,\IZ/q\IZ)\right| = q^{2n^2+n}\prod_{p|q}\prod_{i=1}^n\left(1-\frac{1}{p^{2i}}\right).
  \end{equation}
Let's abbreviate $S_\Gamma(s,d,a) = S$.
The index in $\Sp(2n-2,\IZ)$ of $S$ is the same as the index of $S\pmod{q^2}=:S_{q^2}$ in $\Sp(2n-2,\IZ/q^2\IZ)=:\Gamma_{q^2}$ by the surjectivity of projections of arithmetic groups onto their $(\IZ/m\IZ)$-points and the fact that $S$ contains the principal congruence subgroup modulo $q^2$.
Next we want to simplify the counting by noting that $$|S_{q^2}| =  |S_q||\ker(S_{q^2}\to S_q)|$$
so 
\begin{equation}
  \label{eq:15}
  |\Gamma_{q^2}:S_{q^2}| = \frac{|\Gamma_{q^2}|}{|S_q||\ker(S_{q^2}\to S_q)|}.
\end{equation}
First we treat the cardinality $| S_q|$.
The image $S_q$ consists of matrices of the form
\begin{equation}
  \label{eq:16}
    \left(\begin{array}{cccc}a & x &b&0\\0&1&0&0\\c&y&d&0\\t&z&s&1\\
 \end{array}\right)
\end{equation}
where $(a,b,c,d)\in \Sp(2m-2,\IZ/q\IZ)$ arbitrary, $(x,y,z)\in (\IZ/q\IZ)^{2(m-1)+1}$ arbitrary and $(s,t)\in (\IZ/q\IZ)^{2m-2}$ completely determined by the previous data.
In fact, we can recognize this subgroup as 
\begin{equation}
  \label{eq:17}
  S_q \simeq \Sp(2m-2,\IZ/q\IZ)\ltimes H_{2m-1}(\IZ/q\IZ)
\end{equation}
where $H_k$ is the $k$-dimensional Heisenberg group given by any symplectic form defined over $\IZ$.

Using induction and the obvious cardinality of the Heisenberg group, we have 
\begin{equation}
  \label{eq:18}
  |S_q| = q^{2m-1}q^{2(m-1)^2+(m-1)}\prod_{p|q}\prod_{i=1}^{m-1}\left(1-\frac{1}{p^{2i}}\right)
\end{equation}
which we simplify to 
\begin{equation*}
   |S_q| = q^{2m^2-m}\prod_{p|q}\prod_{i=1}^{m-1}\left(1-\frac{1}{p^{2i}}\right).
\end{equation*}
Now the kernel of the projection morphism consists of symplectic matrices as in \eqref{congr} supplanted by the rest of the congruences giving $A \equiv 1 \pmod{q}$, $B\equiv 0\pmod q$, $C\equiv 0\pmod{q}$ and $D\equiv 1\pmod{q}$ with the extra condition $B_{mm}=0$ in $\IZ/q^2\IZ$.

Writing each such matrix $M$ as $M = I + qX$ where $$X = \mat{X_1}{X_2}{X_3}{X_4}$$ and each of the block matrices with entries in $\IZ/q\IZ$, the symplectic restrictions in $\IZ/q^2\IZ$ become the equations $X_3=X_3^\tr$, $X_2=X_2^\tr$, $X_1 = -X_4^\tr$.
From the congruences in \eqref{congr} the only one not appearing above is $B_{mm}\equiv 0\pmod{q^2}$ which becomes the extra restriction $(X_2)_{mm}=0$.

With this description of the kernel, we can now easily count: $q^{m^2}q^{\frac{m(m+1)}{2}}q^{\frac{m(m+1)}{2}-1} = q^{2m^2+m-1}$ is the cardinality of the kernel.  Now that we have the cardinality of the kernel and the image, we can use \eqref{eq:15} and the fact that $p|q\iff p|q^2$ for a prime $p$ to get the result.
\end{proof}

\section{Orbital spaces and their measures}
\label{sec:orbital-spaces-their}
Using the transformation $T_{(s,d,a)}$ from the previous section, we can translate the space (we emphasize again that these spaces are identified with $F$-tilings inside $\Sp(2n,\IR)$) $$\Sp(2n,\IR)/\stab_\Gamma(s,d,a)\stackrel{T_{(s,d,a)}}{\longrightarrow} \Sp(2n,\IR)/\Gamma_1$$ where $\Gamma_1$ is the congruence subgroup we obtained there.  This transformation takes the representative $(s,d,a)$ to $(s,0,0) = (e^1,sf^1)$ which is fixed by all $\gamma\in \Sp(2m,\IZ)$, for the embedded $\Sp(2m,\IZ) = S_\Gamma((s,0,0))= S_\Gamma((1,0,0))$ in the stabilizer of the plane $\IR e^1+\IR f^1$.
This way we obtain the identity (since $T$ is measure-preserving) 
\begin{eqnarray*}
  \int_{\Sp(2n,\IR)/S_\Gamma(s,d,a)} f(gu_0)f(g v_0)\,d\mu(g) &=& \int_{\Sp(2n,\IR)/\Gamma_1}f(ge^1)f(sgf^1)\,d\mu(g)\\
 &=& [\Sp(2m,\IZ):\Gamma_1]\int_{\Sp(2n,\IR)/\Sp(2m,\IZ)}f(ge^1)f(sgf^1)\,d\mu(g).
\end{eqnarray*}

Now we want to decompose 
\begin{eqnarray*}
& &\int_{\Sp(2n,\IR)/\Sp(2m,\IZ)}f(ge^1)f(sgf^1)\,d\mu(g)\\& =& \int_{\Sp(2n,\IR)/\Sp(2m,\IR)}\int_{\Sp(2m,\IR)/\Sp(2m,\IZ)}f(ghe^1)f(sghf^1)\,d\mu_2(h)\,d\mu_1(g)
\end{eqnarray*}
where $\mu_i$ come from Haar measures on the corresponding group normalized so their product gives $d\mu$.
The most convenient way of doing this is to use the Lie algebra perspective taken in \cite{MoSa}, which is equivalent to the approach of \cite{Gar}.
In both cases, for the embedded $\Sp(2m,\IR)/\Sp(2m,\IZ)$ (which happens to be the exact same space we are considering here) we get the recurrence relation $$\vol(\Sp(2n,\IR)/\Sp(2n,\IZ))=\zeta(2n)\vol(\Sp(2m,\IR)/\Sp(2m,\IZ))$$
so with our normalization $\vol(F)=1$, the latter has total volume $\frac{1}{\zeta(2n)}$.
Since the pair $(e^1,f^1)$ is fixed by this copy of $\Sp(2m,\IR)$, we finally get the expression for the orbital integral
\begin{equation}
  \label{eq:22}
 \int_{\Sp(2n,\IR)/S_\Gamma(s,d,a)} f(gu_0)f(g v_0)\,d\mu(g) = \frac{[\Sp(2m,\IZ):\Gamma_1]}{\zeta(2n)}\int_{\Sp(2n,\IR)/\Sp(2m,\IR)}f(ge^1)f(sgf^1)\,d\mu_1(g).
\end{equation}
The choice of the normalizations made in \cite{MoSa} for the $\mu$ and $\mu_2$ of course forces $\mu_1$ to have a specific normalization, namely the one coming from taking the orthonormal basis vectors of the Lie algebra of $\Sp(2n,\IR)$ orthogonal to those of $\Sp(2m,\IR)$ and tensoring to get the volume form on the quotient.
To understand this space, let $h(x,y) = f(x)f(sy)$, and write the integral as
\begin{equation*}
\int_{\Sp(2n,\IR)/\Sp(2m,\IR)}h(ge^1,gf^1)\,d\mu_1(g).
\end{equation*}
Since $\Sp(2n,\IR)$ is transitive on pairs of the symplectic value $1$ and $\Sp(2m,\IR)$ is the stabilizer of one such point, the homogeneous space can be identified with the hypersurface of pairs $(x,y)\in (\IR^{2n})^2$ with $\langle x,y\rangle=1$ (in fact, it is clear that the pair is the first and $(n+1)$-th columns of $g$).
Now we will give convenient coordinates and write down the measure explicitly.

First, fix $x\in \IR^{2n}\setminus 0$.
Select $y^*(x)$ so that $\langle x,y^*(x)\rangle = 1$ and extend $\Sp(2n,\IR)$-equivariantly 
by defining $y^*(gx) = gy^*(x)$ for all $g\in \Sp(2n,\IR)$.
Since $\Sp(2n,\IR)$ is transitive on non-zero vectors, this defines $y^*$ for all $x$.

Note that if two vectors $y,y'$ satisfy $\langle x,y\rangle =\langle x,y'\rangle = s$, their difference lies on the hyperplane $\langle x,z\rangle= 0$.  On each hyperplane we give coordinates as follows: given $x$ as before, pick a basis $y_1(x),\cdots,y_{2n-1}(x)$ for the hyperplane so that $y_i(gx) = gy_i(x)$.
  
With this definition, it is trivial to check that we get an equivariant assignment of bases for each non-zero $x\in \IR^{2n}$.
Then every element in the hypersurface $\langle x,y\rangle=s$ can be written as 
\begin{equation}
  \label{eq:21}
 (x,y^*(x) + t_1y_1(x)+\cdots+t_{2n-1}y_{2n-1}(x)) = (x,sy^*(x) + t\cdot \tilde{y}(x))
\end{equation}
and this coordinate system is compatible with the diagonal action of $\Sp(2n,\IR)$.
\begin{lemma}
  The measure $dx_1\cdots dx_{2n}\,dt_1\cdots dt_{2n-1}$ coincides with the Haar measure on $\Sp(2n,\IR)/\Sp(2m,\IR)$ with the normalization we chose.
\end{lemma}
\begin{proof}
  This follows directly from the $G$-equivariance of the assignments $y^*$ and $y_i$ and the fact that this is precisely the measure coming from tensoring the Lie algebra directions orthogonal to $\Sp(2m,\IR)$.
\end{proof}
\begin{prop}
\label{prop:4}
We have 
\begin{eqnarray}
  \label{eq:8}
   & &\int_{\Sp(2n,\IR)/S_\Gamma(s,d,a)} f(gu_0)f(g v_0)\,d\mu(g) \nonumber\\ & =&  \frac{[\Sp(2m,\IZ):\Gamma_1]}{s^{2n-1}\zeta(2n)}\int_{\IR^{2n}}\int_{\IR^{2n-1}}f(x)f\left(sy^*+\sum_i t_iy_i(x)\right)\,dt\,dx.
\end{eqnarray}
\end{prop}
\begin{proof}
From the previous lemma, \begin{equation*}
\int_{\Sp(2n,\IR)/\Sp(2m,\IR)}h(ge^1,gf^1)\,d\mu_1(g) = \int_{\IR^{2n}}\int_{\IR^{2n-1}}h(x,y^*(x) + \sum t_iy_i)\,dt\,dx
\end{equation*}
and now substitute the definition and make the change of variables $(x,y)\to (x,sy)$.
\end{proof}

This concludes our analysis of the individual orbital integrals.

\section{Orbital sums in a given symplectic class and effective Ikehara-Wiener}
\label{sec:arithm-coeff-effect}
We now aggregate orbital integrals corresponding to a given $s$ and estimate their contribution.
From the results of the previous sections we have
\begin{eqnarray}\label{eqarr:1}
  & &\sum_{d|s,a\in (\IZ/d\IZ)^*}\int_{\Sp(2n,\IR)/\stab_\Gamma(s,d,a)}f(gu_0,gv_0)\,d\mu(g) \\ &=&\left(\frac{1}{\zeta(2n)s^{2n-1}}\sum_{d|s,a\in (\IZ/d\IZ)^*}     q^{2m+1}\prod_{p|q}\left(1-\frac{1}{p^{2m}}\right)\right)\int_{\IR^{2n}}\int_{\IR^{2n-1}}f(x)f\left(sy^*+\sum_i t_iy_i(x)\right)\,dt\,dx\nonumber\\
&=&\left(\frac{1}{\zeta(2n)s^{2n-1}}\sum_{d|s}\phi(d)     q^{2m+1}\prod_{p|q}\left(1-\frac{1}{p^{2m}}\right)\right)\int_{\IR^{2n}}\int_{\IR^{2n-1}}f(x)f\left(sy^*+\sum_i t_iy_i(x)\right)\,dt\,dx.\nonumber
\end{eqnarray}
The goal in this section is to estimate the contribution of the arithmetic coefficients.
To that end, note that the divisor sum is the Dirichlet convolution of $\phi(s)$ and  $X(s) = s^{2m+1}\prod_{p|s}\left(1-\frac{1}{p^{2m}}\right)$.
\begin{defin}
  Let 
  \begin{equation}
    \label{eq:19}
a(s) = \frac{1}{\zeta(2n)s^{2n-1}}\sum_{d|s}\phi(d)q^{2m+1}\prod_{p|q}\left(1-\frac{1}{p^{2m}}\right) = \frac{(\phi*X)(s)}{s^{2n-1}\zeta(2n)}
  \end{equation}
denote the $s$-th arithmetic coefficient in the integral above, and $$A(M) = \sum_{s\leq M}a(S)$$ its summatory function.
\end{defin}
\begin{lemma}
  The $L$-function corresponding to $\phi*X$ is $L(s)=\frac{\zeta(s-2n+1)}{\zeta(s)}$.
\end{lemma}
\begin{proof}
  This follows from a direct computation of the Euler product of $\phi*X$.
Since the $L$-series of $\phi$ is $\frac{\zeta(s-1)}{\zeta(s)}$, we need to show the corresponding one for $X$ is $\frac{\zeta(s-2n+1)}{\zeta(s-1)}$.
We summarize the computation of the Euler product for $X$:
For a fixed prime power $p^k$, we have $X(p^k)= X(p) = p^{2n-1}\left(1-p^{-2n+2} \right)$.
Thus the Euler factor at $p$ is 
$$\left(1+X(p)(p^{-s}+p^{-2s}+\cdots)\right) = \left(1+p^{2n-1-s}\left(1-p^{-2n+2}\right)\left(1+p^{-s}+\cdots\right)\right)$$ 
and summing the geometric series we get
$$\left(1+p^{2n-1-s}\left(1-p^{-2n+2}\right)\left(1 - p^{-s}\right)^{-1}\right) $$
which after simplifying gives
$$1+ \frac{1-p^{-2n+2}}{1-p^{-s+2n-1}}p^{-s+2n-1} = \frac{1-p^{-s+1}}{1-p^{-s+2n-1}}$$
which we recognize as the $p$-th Euler factor of 
\begin{eqnarray*}
 L(X,s) = \prod_p\frac{1-p^{-s+1}}{1-p^{-s+2n-1}}=  \frac{\zeta(s-2n+1)}{\zeta(s-1)}.
\end{eqnarray*}
\end{proof}
The next observation, although trivial, is very important in handling the terms $\frac{\phi*X}{s^{2n-1}}$.
\begin{lemma}\label{bnd}
  The convolution $\phi*X$ is always bounded by $s^{2n-1}$.
\end{lemma}
\begin{proof}
  The function is multiplicative, and for prime powers the claim is clear.
\end{proof}
We now need an estimate on the summatory function of $\phi*X$; since its $L$-function is explicitly given in terms of the Riemann zeta function, we can use classical methods to give an Ikehara-Wiener type estimate with error term.
Tauberian theorems for Dirichlet series are usually either given with no explicit error estimate or in such generality that it is easier to argue directly from the effective Perron formulae, which we do next.
\begin{lemma}
  Let $S(x)$ be the summatory function of $s_m= (\phi*X)(m)$.
We have for every $\epsilon>0$ sufficiently small
\begin{equation}
  \label{eq:28}
S(x) = \frac{x^{2n}}{2n\zeta(2n)}+O_\epsilon\left(\frac{x^{2n}\ln(x)}{x^{\frac{3}{13}-\epsilon}}\right).
\end{equation}
\end{lemma}
\begin{proof}
  First, we use the effective Perron formula \cite[Section II.2.1, Theorem 2]{Ten} to write, for $T$ large enough and $k>2n$,
  \begin{equation}
    \label{eq:29}
    S(x) = \int_{k-iT}^{k+iT}L(s)\frac{x^s}{s}\,ds + O\left(x^k\sum_{m=1}^\infty \frac{s_m}{m^k(1+T|\log(x/m)|}\right).
  \end{equation}
Using Lemma \ref{bnd} and making the choice $k = 2n+\epsilon$, we get 
\begin{equation*}
    S(x) = \frac{1}{2\pi i}\int_{k-iT}^{k+iT}L(s)\frac{x^s}{s}\,ds + O\left(x^{2n+\epsilon}\sum_{m=1}^\infty \frac{1}{m^{1+\epsilon}(1+{T|\log(x/m))|})}\right).  
\end{equation*}
Now by elementary calculus $$\sum_{m=\sqrt{x}}^{\infty}\frac{1}{m^{1+\epsilon}} = O_\epsilon(x^{-\epsilon/2})$$
and for $1\leq m\leq \sqrt{x}$ we have $1+T|\log(x/m)|>\frac{1}{2}T\log(x)$
giving overall
\begin{equation*}
    S(x) = \frac{1}{2\pi i}\int_{k-iT}^{k+iT}L(s)\frac{x^s}{s}\,ds + O\left(\frac{x^{2n+\epsilon}}{T\log(x)}\right).  
\end{equation*}
Next, we shift the integration from the line segment $\Re(s)=2n+\epsilon,\,|\Im(s)|\leq T$ to $\Re(s) = 2n-\frac{1}{2}$ passing through $s=2n$ where the numerator of $L(s)$ has a simple pole of residue $\frac{1}{\zeta(2n)}$ (the denominator has no zeros in this rectangle so there are no poles contributing from it).
Therefore, from Cauchy's theorem we get 
\begin{eqnarray*}
  & &S(x) = \frac{x^{2n}}{2n\zeta(2n)}\\ &+& \frac{1}{2\pi i}\int_{2n-\frac{1}{2}-iT}^{2n-\frac{1}{2}+iT}L(s)\frac{x^s}{s}\,ds + \frac{1}{2\pi i}\int_{2n-\frac{1}{2}-iT}^{2n+\epsilon-iT}L(s)\frac{x^s}{s}\,ds +\frac{1}{2\pi i}\int_{2n-\frac{1}{2}+iT}^{2n+\epsilon+iT}L(s)\frac{x^s}{s}\,ds\\&+& O\left(\frac{x^{2n+\epsilon}}{T\log(x)}\right).  
\end{eqnarray*}
Since $n\geq 1$, $\frac{1}{\zeta(s)}$ is bounded on the half-plane $\Re(s)\geq 2n-\frac{1}{2}$ by $\prod_p(1+p^{-3/2})$ and $|s|\geq |\Re(s)|\geq 3/2$.
Thus $|L(s)|$ is determined on this half-plane by the behavior of $\zeta(s-2n+1)$ which is known:
on the critical line we have the subconvexity bound $|\zeta(\frac{1}{2}+i\tau)|\leq C\tau^{1/6}\log(\tau)$ (\cite[Section II.3.4, Corollary 5.2]{Ten}) with obvious conventions for small $\tau$, so $$\left|\frac{1}{2\pi i}\int_{2n-\frac{1}{2}-iT}^{2n-\frac{1}{2}+iT}L(s)\frac{x^s}{s}\,ds\right| = O\left(x^{2n-\frac{1}{2}}T^{\frac{7}{6}}\log(T)\right). $$

For the horizontal integrals, at levels $\pm iT$ we have $|\zeta(s)|\leq T^{\frac{1-\sigma}{3}}$ for $\frac{1}{2}\leq \sigma\leq 1$ and beyond $1$ $\zeta$ is $O(\log(T))$ so taking the worse bound over the entire line segment and using $|s|\geq T$, with the change of variables $s' = 2n-1+\sigma+iT$ we get
$$\int_{2n-\frac{1}{2}+iT}^{2n+\epsilon_iT}L(s)\frac{x^s}{s}\,ds \leq C \frac{x^{2n-1}}{T}T^{1/3}\int_{\frac{1}{2}}^{1+\epsilon}T^{-\sigma/3} x^\sigma\,d\sigma$$
$$\leq C\frac{x^{2n+\epsilon}}{T^{-\epsilon/3-1}}\log(x/T^{1/3}).$$

Finally, looking at the contributions of the various segments we let $T=x^a$ for an $a$ to be optimized: we have to balance $2n+\epsilon - a$ from the horizontal segments and the Perron error with $2n-\frac{1}{2}+7/6a$ from the critical line so we choose $2n-a=2n-\frac{1}{2}+7/6a$ to get $a=\frac{3}{13}$.
\end{proof}
\begin{rem}
  The above result is not optimal and the exponent $3/13$ can be improved with a more careful analysis.
Below we will use a softer bound implied by the lemma to make computations easier.
\end{rem}
The next proposition gives the asymptotics of the summatory function of the arithmetic coefficient in the expression we got for the second moment
\begin{prop}\label{asym}
  We have the asymptotic
  \begin{equation}
    \label{eq:23}
    A(M) = \frac{M}{\zeta(2n)^2}(1+E(M))
  \end{equation}
with error term $E = O(M^{-\delta})$ for some $\delta>0$.
\end{prop}
\begin{proof}
By the previous lemma, the summatory function $S(M) = \sum_{s=1}^M\phi*X$ has asymptotic $$S(M) = \frac{M^{2n}}{2n\zeta(2n)}(1+O(M^{-\delta})).$$
Then a summation by parts gives
  \begin{eqnarray*}
\sum_{s=1}^M\frac{\phi*X(s)}{s^{2n-1}}&=& 
\sum_{s=1}^M\frac{dS(s)}{s^{2n-1}}\\ &=& 
\sum_{s=1}^M S(s)d(s^{-2n+1})+\frac{S(M)}{M^{-2n+1}}+O(1)\\ &=&
% \frac{1}{\zeta(2n)}\sum_{s=1}^M S(s)d(s^{-2n+1})\\ &=&
\sum_{s=1}^M \frac{1}{\zeta(2n)}\frac{s^{2n}}{2n}(2n-1) s^{-2n} +\frac{M}{\zeta(2n)}+O(M^{1-\delta}) \\&=&
 \frac{1}{\zeta(2n)}(M+O(\int_1^Ms^{-\delta}\,ds)) = \frac{M}{\zeta(2n)}(1+E(M)).
  \end{eqnarray*}

\end{proof}
\section{Stieltjes integration and discrete Fubini}
\label{sec:stieltj-integr-discr}
In this section we will give an explicit formula for the mean square error averaged over $\Sp(2n,\IR)/\Sp(2n,\IZ)$; we will not use it verbatim in the sequel but consider it important because it clarifies what precisely needs to be bounded and in what way.  It should be noted that this formula could in principle be used to give discrepancy bounds over $\Sp(2n,\IR)/\Sp(2n,\IZ)$ rather than $\GSp$, but he could not control the terms in sufficient generality.

To begin, observe that the integral $$\int_{\IR^{2n}}\int_{\IR^{2n-1}}f(x)f\left(sy^*+\sum_i t_iy_i(x)\right)\,dt\,dx$$
has the property that if we integrate over $s$ and use Fubini's theorem, we get $\vol(B)^2$.
This is because the inner integration will be over $\IR^{2n}$ of the characteristic function $f$ with respect to Lebesgue measure: $y^*$ and the $y_i$ together form a basis for $\IR^{2n}$ which comes from the standard basis $e^i,f^i$ by an element $g\in \Sp(2n,\IR)$ which has determinant $1$.

Using equation \eqref{eqarr:1}, the second moment of $\widehat{f}$ becomes
\begin{equation}
  \label{eq:20}
  \sum_{\trm{lin. dep.}} + \sum_{s=-\infty}^\infty a(|s|)\int_{\IR^{2n}}\int_{\IR^{2n-1}}f(x)f\left(sy^*+\sum_i t_iy_i(x)\right)\,dt\,dx.
\end{equation}
The first term gathers the orbits of linearly dependent pairs of primitive vectors, which will clearly have contribution asymptotic to $\frac{\vol(B)}{\zeta(2n)}$.
Thus we focus on the second term over primitive pairs, as it emerges after the reductions in the previous sections.
Let's give a convenient name to the geometric term in the sum:
\begin{defin}
  Let 
  \begin{equation}
    \label{eq:24}
    G(s) = \int_{\IR^{2n}}\int_{\IR^{2n-1}}f(x)f\left(sy^*+\sum_i t_iy_i(x)\right)\,dt\,dx
  \end{equation}
be the integral over the symplectic hypersurface, and 
  \begin{equation}
    \label{eq:25}
    C(f) = \sum_{s=-\infty}^\infty a(|s|)G(s)
  \end{equation}
the total contribution of the primitive pairs.
\end{defin}
We will think of $C(f)$ as a twisted integral by the arithmetic terms $a(s)$ and our goal is to decouple the two quantities.
In order to isolate what we expect to be the main term, write $C(f)$ as a Stieltjes integral
\begin{equation}
  \label{eq:26}
  C(f) = \int_\IR a(|s|)G(s)d[s] = \int_\IR a(|s|)G(s)\,ds + \int_\IR a(|s|)G(s)\,d\{s\}
\end{equation}
and isolate the contribution from $G$:
\begin{equation*}
  C(f) = \int_\IR G(s)\,ds - \int_\IR (1-a(|s|))G(s)\,ds + \int_\IR a(|s|)G(s)\,d\{s\}
\end{equation*}
giving
\begin{equation*}
  C(f) = \frac{\vol(B)^2}{\zeta(2n)^2} - \int_\IR (1-a(|s|))G(s)\,ds + \int_\IR a(|s|)G(s)\,d\{s\}.
\end{equation*}
Here it is crucial that $a(|s|)$ and $1-a(|s|)$ lie in $[0,1]$.
We summarize the result in the following
\begin{prop} We have the explicit formula for the error term
  \begin{equation}
    \label{eq:27}
    \left|C(f)-\frac{\vol(B)^2}{\zeta(2n)^2}\right| = \int_\IR a(|s|)G(s)\,d\{s\} - \int_\IR (1-a(|s|))G(s)\,ds.
  \end{equation}
\end{prop}

We conclude this section by noting that 
\begin{equation}
  \label{eq:39}
  G(s) = \int_{\IR^{2n}}f(x)\tilde{f}(r\omega, s)\,dx
\end{equation}
where $x=r\omega$ is the polar decomposition of $x$ and $\tilde{f}$ is the Radon transform of $f$.  
We expect that the realization of $G(s)$ as an averaged Radon transform will be instrumental in giving bounds for $C(f)$. 
\section{Discrepancy in the symplectic cone}
\label{sec:discr-sympl-cone}

At the heart of \cite{Schm} is a mean square error estimate for the lattice points taken from a random lattice of determinant $0< \det(\Lambda)\leq 1$.
That estimate was used to make a claim about almost everywhere bounds with respect to Lebesgue measure in $\IR^{n^2}$, by giving explicit bounds for an equivalent (mutually absolutely continuous) measure.
Here, we will prove similar bounds averaging over a much thinner set, namely the subgroup $\GSp^+(2n,\IR)<\trm{GL}(2n,\IR)$ given by $$\GSp^+(2n,\IR) = \{g\in\trm{GL}(2n,\IR): gJg^\tr=c_gJ\trm{ for some }c_g>0\}$$ where $J$ is the matrix in the definition of $\Sp$.
We can show that $c_g = \det(g)^{\frac{1}{n}}$ and observe that $c_g^{-1/2}g\in \Sp(2n,\IR)$.

The method is nearly identical to that of the aforementioned work once we have the expression \eqref{eq:25} at our disposal.
For this reason, we will be brief in our arguments, and only highlight the main ideas and points of departure from \cite[Section 4]{Schm}.
One important difference concerns the right scaling factors; the determinant scaling needs to be replaced by more intrinsic factor in higher symplectic spaces.

First of all, note $\GSp^+(2n,\IR) = \Sp(2n,\IR)\ltimes \IR^*$ where we write $\IR^* = \IR_{>0}$ viewed as the connected component of the identity in the multiplicative group of $\IR$.
Under the map $g\mapsto (c_g^{-\frac{1}{2}}g,c_g)$, the Haar measure of $\GSp^+$ decomposes as $dh(g) = \nu\,d\mu(\tilde{g})\,dm(\nu)$ which we normalize so that $d\mu$ is the probability measure on $\Sp(2n,\IR)$ and $dm$ the usual Lebesgue measure giving measure $1$ to $(0,1]$; in particular the measure of the unit cone $C = \{g\in\GSp(2n,\IR): c_g\in (0,1]\}$ is $\frac{1}{2}$.
Therefore, integrating against Haar measure on $C$ will not give an average of $\frac{1}{\zeta(2n)}\vol(B)$ lattice points in a Borel set $B$.
Schmidt's idea in the $n=1$ case was then to change the measure so that the average over $\SP$ will be the same as the average over $C$ for an appropriately weighted version of $\chi_B$.
To that end, define $du$ on $\GSp(2n,\IR)/\GSp(2n,\IZ)$ by $$\int_C f(g)du(g) = \int_0^1 \int_\SP h(\nu^{\frac{1}{2}}\tilde(g))d\mu(\tilde{g})d\nu$$ for a Borel function $h$ on $C$.

The function to which we will apply this integration (in the Cartesian coordinates mentioned above) is $h(\tilde{g},\nu) = \nu^n \widehat{f\circ L_{\nu^{\frac{1}{2}}}}(\tilde{g})$ where $L_t$ is left multiplication by $t$.
Note that $h(\tilde{g},\nu)$ is simply $$\det(g)\sum_{\lambda\in g\IZ^{2n}_\prim}f(g)\quad\trm{ for } g\in C$$ in the coordinates introduced above.

This way 

$$ \int_C h du = \int_0^1 \nu^n \int \widehat{f\circ L_{\nu^{\frac{1}{2}}}}(g) d\mu(\tilde{g})\,d\nu$$
$$ = \frac{1}{\zeta(2n)}\int_0^1\nu^n \int_{\IR^{2n}}f\circ L_{\nu^{\frac{1}{2}}}(x)dx\,d\nu$$
$$ = \frac{1}{\zeta(2n)}\int_0^1\nu^n \int_{\IR^{2n}} \nu^{-n}f(x)dx\,d\nu $$
$$ = \frac{1}{\zeta(2n)}\int_{\IR^{2n}}f(x)\,dx$$
recovering the correct statistics when $f=\chi_B$.

From now on we will drop the tilde from a unimodular $g$ since we will always use the coordinates introduced above for $C$ and $\GSp(2n,\IR)$ in general.
For the second moment of $h$, we have

\begin{prop}\label{main_bound}
  Let $f$ be the characteristic function of a Borel set $B$ of volume $V$ such that there exists $\epsilon>0$ such that

  \begin{equation}\label{main_assum}
 \int_{B}\int_B|\langle x,y\rangle|_+^{-\delta} \,dx\,dy\leq CV^{2-\epsilon}.
  \end{equation}
  We have the mean square bound 
  \begin{equation}
    \label{eq:30}
  \int_C \left|h(g,\nu)-\frac{\int f}{\zeta(2n)}\right|^2\,dg\,d\nu\ll \left(\int f\right)^{2-\epsilon}.    
  \end{equation}
\end{prop}

\begin{proof}
  Following the unwrapping procedure as before, the second moment over $C$ becomes 
$$\sum_{s=-\infty}^\infty a(|s|)\tilde{G}(s) $$
where $$\tilde{G}(s) = \int_0^1\nu^{2n}\int_{\IR^{2n}}f(\nu^{1/2}x)\int_{\IR^{2n-1}}f\left(s\nu^{1/2}y^*(x)+\sum_i t_i\nu^{1/2}y_i(x)\right)\,dt\,dx\,d\nu.$$

Now make the change of variables $x' = \nu^{1/2}x$ and let $y'_i = \nu^{1/2}y_i$; note $y'^* = \nu^{-1/2}y^*$ modulo the hyperplane $\langle x',y'\rangle=0$.
This way we get 
$$\tilde{G}(s) = \int_0^1\nu^n \int_{\IR^{2n}}f(x')\int_{\IR^{2n-1}}f\left(s\nu y'^*(x')+\sum_i t_iy'_i(x')\right)\,dt\,dx'\,d\nu.$$

Taking the summation over $s$ inside the integral, making the substitution $u=s\nu$ and using Fubini's theorem, we get
\begin{equation}
  \label{eq:31}
  \sum_{s}a(|s|)\tilde{G}(s) = \int_{\IR^{2n}}f(x')\int_{\IR^{2n-1}}\int_\IR f\left(u y'^*(x')+\sum_i t_iy'_i(x')\right) u^{n-1}\left(\sum_{s>|u|}\frac{a(s)}{s^n}\right)\,du\,dt\,dx.
\end{equation}
An argument similar to that of Proposition \ref{asym} shows that $$ u^{n-1}\sum_{s>|u|}\frac{a(s)}{s^n} = \frac{1}{\zeta(2n)^2} + O(\min\{1,u^{-\delta}\})$$ and so 
\begin{eqnarray*}
    \sum_{s}a(|s|)\tilde{G}(s) &=& \frac{1}{\zeta(2n)^2}\int_{\IR^{2n}}f(x')\int_{\IR^{2n-1}}\int_\IR f\left(u y'^*(x')+\sum_i t_iy'_i(x')\right) \,du\,dt\,dx\\ &+& O\left(\int_{\IR^{2n}}f(x')\int_{\IR^{2n-1}}\int_\IR u^{-\delta}f\left(u y'^*(x')+\sum_i t_iy'_i(x')\right) \,du\,dt\,dx\right)\\
&=& \left(\frac{\int f}{\zeta(2n)}\right)^2 + O\left(\int_{\IR^{2n}}f(x')\int_{\IR^{2n-1}}\int_\IR u^{-\delta}f\left(u y'^*(x')+\sum_i t_iy'_i(x')\right) \,du\,dt\,dx\right).
\end{eqnarray*}
The error term can be written
\begin{equation*}
 E = \int_{\IR^{2n}}\int_{\IR^{2n}}|\langle x,y\rangle|_+^{-\delta}f(x)f(y) \,dx\,dy
\end{equation*}
which completes the proof.

\end{proof}

\begin{cor}

Let $f$ be the characteristic function of a Borel set $B$ of volume $V$ such that all but $V^{\epsilon'}$ volume lies in a set of the form $gB(0,V^{\frac{1}{2n}})$ for some $g\in\Sp(2n,\IR)$. 
We have the mean square bound for some $\epsilon>0$ ( dependent only on the dimension $n$ and $\epsilon'$):
  \begin{equation}  
  \int_C \left|h(g,\nu)-\frac{\int f}{\zeta(2n)}\right|^2\,dg\,d\nu\ll \left(\int f\right)^{2-\epsilon}.    
  \end{equation}

\end{cor}
\begin{proof}
  This is a straightforward consequence of the following general lemma:

  \begin{lemma}
    Let $S$ be an arbitrary Borel set of finite volume and $B$ a ball centered at the origin.\footnote{This lemma generalizes significantly: for well behaved sets $B$ bound $\chi_B$ by smooth compactly supported $f$ losing a $m(B)^\epsilon$ measure, and then approximate $f$ by $SO(2n)$-finite vectors in a Sobolev space. The arguments in the proof can be used to bound the error as long as the Sobolev approximation is good enough.}  Assume that $m(S)\geq m(B)\geq m(B(0,1))$.  Then for some dimensional constant $C>0$, 
    \begin{equation}
\int_{S}\int_{B}|\langle s,y\rangle|_+^{-\delta} \,dy\,ds\leq Cm(S)m(B)^{1-\frac{\delta}{2n}}.
\end{equation}
\end{lemma}
\begin{proof}
  For each $s\in S$ choose $r_s\in SO(2n)$ such that $r_s(\tau(s)) = \|s\|e_1$.  
  By rotational invariance of the inner product, we have
  $$ \int_{S}\int_{B}|\langle s,y\rangle|_+^{-\delta} \,ds\,dy = \int_{S}\int_{B}|[ r_s(\tau(s)),r_s(y)]|_+^{-\delta} \,dy\,ds$$
  $$ = \int_{S}\int_{r^{-1}_s(B)}|[ \|s\|e_1,y]|_+^{-\delta} \,dy\,ds$$
  $$ = \int_{S\cap \{\|s\|\geq 1\}} \|s\|^{-\delta}\int_{B\cap \{|y_1|\geq 1\}}|y_1|_+^{-\delta}\,dy\,ds + O(m(B)).$$

  Now $\int_B|y_1|_+^{-\delta} \leq C m(B)^{1-\frac{\delta}{2n}}$ which concludes the lemma using the trivial bound $\|s\|^{-\delta}\leq 1$.
\end{proof}
\end{proof}

\section{From mean square error to metrical bounds}
\label{sec:from-mean-square}
In this section we summarize the transition from Proposition \ref{main_bound} to the discrepancy bound of \eqref{eq:3}; the method is identical to Schmidt's arguments in \cite[Lemmata 12 and 13, Theorem 1]{Schm} apart from different exponents, but since we do not aim for optimality in the bounds, our exposition will be much simpler.  We will abuse notation and write $\delta, \epsilon,\epsilon',\cdots$ for several different constants, all linearly related to the original ones to avoid proliferation of symbols.  Finally, we assume at the onset that $B_t$ is a continuous family of Borel sets with the hypotheses in Proposition \ref{main_bound} and such that the volumes $\vol(B_t)=v_t$ take all large real values.
\begin{thm}
  Let $B_t$ be a nested family of Borel sets satisfying the properties of Proposition \ref{main_bound}.  Then there exists a $\delta>0$ depending only on the dimension and the exponent $\epsilon$ in \ref{main_bound} so that for almost all $\Lambda \in \GSp^+(2n,\IR)/\GSp^+(2n,\IZ)$ we have 
  \begin{equation}
    \label{eq:38}
    D(B_t,\Lambda)  \leq C(\Lambda, n,\delta)v_t^{-\delta}.
  \end{equation}
\end{thm}
From the bound 

$$\int_C \left(\det(g)\widehat{\chi_B}(g\IZ^2) - \frac{\vol(B)}{\zeta(2n)}\right)^2d\mu(g)\leq C\vol(B)^{2-2\delta}$$
we get via Chebychev's inequality
\begin{equation}
  \label{eq:33}
  \mu\left(\Lambda\in C: \left|\det(\Lambda)\cdot (\#\Lambda_\prim\cap B) - \frac{\vol(B)}{\zeta(2n)}\right| > k\vol(B)^{1-\delta}\right)\leq \frac{1}{k^2}
\end{equation}
which we rearrange as 
\begin{equation}
  \label{eq:34}
  \mu\left(\Lambda\in C: \left|\frac{(\#\Lambda_\prim\cap B)}{\vol(B)} - \frac{1}{\det(\Lambda)\zeta(2n)}\right| > k\vol(B)^{-\delta}\right)\leq \det(\Lambda)^2\frac{1}{k^2};
\end{equation}
with $k=\vol(B)^{\frac{\delta}{2}}$ and the definition of the discrepancy function, we get
\begin{equation}
  \label{eq:35}
    \mu\left(\Lambda\in C: D(B,\Lambda) > \vol(B)^{-\delta/2}\right)\leq \det(\Lambda)^2\frac{1}{\vol(B)^\delta}.
\end{equation}
Now consider the family $B_t$ and let $l\in\IZ^+$ large.  Choose a subsequence $B_{m(l)}=B'_l$ of volumes $v(l) = l^{\frac{1+\rho}{\delta}}$ for any fixed $\rho$ with $0<\rho<1$.

Then by Borel-Cantelli we see that for almost all lattices $\Lambda\in C$, we have 
\begin{equation}
  \label{eq:36}
  D(B'_l,\Lambda) \leq C(\Lambda,n)\vol(B'_l)^{-\delta/2}
\end{equation}
for some constant $C(\Lambda,n)$ depending only on the (covolume of) the lattice and the dimension.

For the remaining elements $B_t$ in the sequence, choose $l$ such that $$l^{\frac{1+\rho}{\delta}}\leq v_t\leq (l+1)^{\frac{1+\rho}{\delta}}$$ and use the inclusions to interpolate the discrepancy of $B_t$ between the two discrepancies $D(B'_l,\Lambda)$ and $D(B'_{l+1},\Lambda)$:
\begin{eqnarray*}
D(B_t,\Lambda)  &\leq &  \max(\left(\frac{l+1}{l}\right)^{\frac{1+\rho}{\delta}}
\left[D(B'_l,\Lambda) + \left(\left(\frac{l}{1+l}\right)^{\frac{1+\rho}{\delta}}-1\right)\right],\\ & &
 \left(\frac{l+1}{l}\right)^{\frac{1+\rho}{\delta}}\left[D(B'_{l+1},\Lambda) + \left(\left(\frac{l+1}{l}\right)^{\frac{1+\rho}{\delta}}-1\right)\right]).
\end{eqnarray*}
Note that the coefficients are bounded by $2^{\frac{1+\rho}{\delta}}$ and the additive error is of the order $l^{-1}$; since $\rho<1$, we see $l^{\frac{2}{\delta}}>l^{\frac{1+\rho}{\delta}}=\vol(B'_l)$ so $$l^{-1} = (l^{\frac{2}{\delta}})^{-\delta/2}< \vol(B'_l)^{\frac{-\delta}{2}}$$
giving 
\begin{equation}
  \label{eq:40}
  D(B_t,\Lambda) \leq 2^{\frac{1+\rho}{\delta}+1}\vol(B'_l)^{-\delta/2} \leq C(\Lambda,n,\delta)v_t^{-\delta/2}.
\end{equation}
 as $t\to \infty$.

Finally, consider the cone $C_L$ of lattices $\Lambda\in \GSp(2n,\IR)/\GSp(2n,\IZ)$ with $0<\det(\Lambda)\leq L$ for $L\in \IN$.  Then $\frac{1}{L^{1/n}}\Lambda$ lies in the unit cone; applying the results above to the dilated family $L^{-\frac{1}{n}}B_t$ which again satisfies all the hypotheses as the original $B_t$ we get that for almost every $\Lambda \in C_L$ we have the same discrepancy bounds as before (the $\delta$ does not change); taking the intersection of these countable co-null sets in $C_L$ for $L=1,2,\cdots$ we see that for almost every lattice in $\GSp^+(2n,\IR)/\GSp(2n,\IZ)$ we have $D(B_t,\Lambda) = O(\vol(B_t)^{-\delta})$ for some $\delta>0$.

\end{document}